\documentclass[12pt]{amsart}
\usepackage{fullpage,url,amssymb,enumerate,colonequals}

\usepackage[utf8]{inputenc}

\usepackage{mathrsfs} 
\usepackage{MnSymbol}
\usepackage{extarrows}
\usepackage{lscape}
\usepackage[all,cmtip]{xy}

\usepackage{amscd}

\usepackage[OT2,T1]{fontenc}

\usepackage{color}

\usepackage[
        colorlinks, citecolor=darkgreen,
        backref,
        pdfauthor={Imin Chen}, 
]{hyperref}

\usepackage{comment}

\usepackage[ruled,vlined]{algorithm2e}

\usepackage{listings}

\newcommand{\F}{\mathbb{F}}

\newcommand{\Q}{\mathbb{Q}}
\newcommand{\R}{\mathbb{R}}
\newcommand{\Z}{\mathbb{Z}}

\newcommand{\rhobar}{{\overline{\rho}}}

\newcommand{\calO}{\mathcal{O}}

\newcommand{\Fp}{\mathfrak{p}}

\DeclareMathOperator{\Frob}{Frob}
\DeclareMathOperator{\Gal}{Gal}

\DeclareMathOperator{\Norm}{Norm}

\DeclareMathOperator{\Res}{Res}

\DeclareMathOperator{\tr}{tr}

\newcommand{\GL}{\operatorname{GL}}

\providecommand{\case}[1]{\begin{cases} #1 \end{cases}}

\makeatletter
\renewcommand*\env@matrix[1][*\c@MaxMatrixCols c]{%
  \hskip -\arraycolsep
  \let\@ifnextchar\new@ifnextchar
  \array{#1}}
\makeatother

\providecommand{\GL}{\text{GL}}

\providecommand{\inv}{^{-1}}

\providecommand{\Res}{\text{Res}}

\renewcommand{\bar}[1]{\overline{#1}}
\renewcommand{\mod}[1]{\text{ (mod $#1$)}}

\providecommand{\N}{\mathbb{N}}
\providecommand{\R}{\mathbb{R}}
\providecommand{\Z}{\mathbb{Z}}
\providecommand{\Q}{\mathbb{Q}}

\renewcommand{\S}{\mathbb{S}}

\providecommand{\F}{\mathbb{F}}

\newcommand{\rad}{\text{rad}}

\theoremstyle{definition}

\newtheorem{theorem}{Theorem}[section]
\newtheorem{corollary}[theorem]{Corollary}
\newtheorem{remark}[theorem]{Remark}

\newtheorem{lemma}[theorem]{Lemma}
\newtheorem{prop}[theorem]{Proposition}
\newtheorem{defn}[theorem]{Definition}
\newtheorem{ex}[theorem]{Example}

\newtheorem{conj}[theorem]{Conjecture}

\definecolor{darkgreen}{rgb}{0,0.5,0}

\title{Sch\"affer's conjecture for even $k \le 42$}

\date{\today}

\thanks{I.C.\ is supported by NSERC Discovery Grant RGPIN-2023-03457. S.C.\ is supported by NSERC Discovery Grant RGPIN-2025-04264. J.H. is supported by an NSERC CGRS award.} 

\begin{document}

\begin{abstract} 
In 1956, Sch\"affer made a conjecture on the whole number solutions to the Diophantine equation $S_k(x):=1^k+2^k+\cdots+x^k=y^n$. In this paper we show how for even $k$, we can descend to a long list of three coupled equations of the form $Ax^p-By^p=1$, where $A,B\in\Q$. To this system, we apply the modular method, linear forms in logarithms, and local methods. As a result, we prove Sch\"affer's conjecture for even $k \le 42$.
\end{abstract}

\author{Imin Chen}

\address{Department of Mathematics, Simon Fraser University\\
Burnaby, BC V5A 1S6, Canada.} 
\email{ichen@sfu.ca}

\author{Stephen Choi}

\address{Department of Mathematics, Simon Fraser University\\
Burnaby, BC V5A 1S6, Canada.} 
\email{schoia@sfu.ca}

\author{James Houle}

\address{Department of Mathematics, Simon Fraser University\\
Burnaby, BC V5A 1S6, Canada.} 
\email{james\_houle@sfu.ca}


\maketitle

{
  \hypersetup{linkcolor=black}
  \tableofcontents
}

\section{Introduction}

Sch\"affer's conjecture \cite{schaeffer} concerns the following diophantine problem:

\begin{conj}
    \textit{(Sch\"affer's Conjecture)} Let $S_k(x) = 1^k + 2^k + \cdots + x^k$. Then the only solutions to the equation
    \begin{equation}
        \label{main}
        S_k(x) = y^n
    \end{equation}
    in $x, y, k, n \in \N$ and $n \ge 2$ is the trivial solution $(x,y,k,n) = (1,1,k,n)$ unless
    \begin{equation}
        \label{infinite-families}
        (k,n) \in \left\{ (1,2), (3,2), (3,4), (5,2) \right\},
    \end{equation}
    in which case there are infinitely many solutions in $x,y \in \N$, or 
    \[
        (x,y,k,n) = (24,70,2,2).    
    \]
\end{conj}

Sch\"affer showed that for fixed $(k,n)$ not in \eqref{infinite-families}, there are only finitely many solutions to \eqref{main}. However his method was ineffective \cite{schaeffer}. Since then, there have been many partial results towards proving the conjecture. In 1980, Gy\"ory, Tijdeman, and Voorhoeve \cite{GTV} provided an effective proof showing that for fixed $k\notin\{3,5\}$, \eqref{main} has only finitely many solutions. Using the fact that $S_k(x)$ has a closed form in terms of Bernoulli polynomials, we can write
\begin{equation*}
    S_k(x)=\case{
        \frac{1}{C_k}x^2(x+1)^2T_k(x) & k\text{ odd},\\
        \frac{1}{C_k}x(x+1)(2x+1)T_k(x) & k\text{ even},
    }
\end{equation*}
where $C_k\in\N$ and $T_k(x)\in\Z[x]$. Leveraging this structure, many authors have verified cases of Sch\"affer's conjecture by fixing values of $k$ and showing that there are no solutions for some or all $n$. Jacobson, Pint\'er, and Walsh \cite{JPW} verified the conjecture for even $2\le k\le 58$ and $n=2$ and Pint\'er \cite{Pinter} verified the conjecture for odd $1\le k<170$ and even $n$. One of the most complete results is due to Bennett, Gy\"ory, and Pint\'er \cite{BGP}, who proved the conjecture for each $1\le k\le 11$.

Taking a different approach, Hajdu \cite{Hajdu} verified the conjecture for certain congruence classes of $x$, for example, when $k=1$ or $k$ is even, and $x\equiv3,4\mod8$, then \eqref{main} has no solutions. Building on this work, B\'erczes, Hajdu, Miyazaki, and Pink \cite{BHMP} settled the conjecture when $x<25$.

In the work of \cite{BGP}, the factors $x$ and $x+1$ of $S_k(x)$ play a crucial role, allowing \eqref{main} to descend to a generalized Fermat equation of the form $Ax^n-By^n=1$ for $A,B\in\N$. We take this idea further and, for even values of $k$, descend \eqref{main} to a system of three such generalized Fermat equations with rational coefficients, rather than integer coefficients.

There is a large collection of tools that one can use to study such equations, for example, the theory of linear forms of logarithms introduced by Baker \cite{Baker} and the modular method, which became a staple ingredient for solving Diophantine equations after \cite{Wiles} \cite{TaylorWiles1995} \cite{BCDT2001} showed that all rational elliptic curves are modular.

It suffices to consider prime exponents $n = p$. For even $k$, we perform a detailed descent argument to a long list of coupled systems of three equations of the form
\begin{align*}
    A_0x^p-B_0y^p &= 1 \\
    A_1z^p-B_1y^p &= 1 \\
    A_2x^p-B_2z^p &= 1 
\end{align*}
where $A_i,B_i\in\Q$ and each coupled system represents a case in the descent. We generalize Kraus' recipe for computing the Serre level \cite{Kraus} in signature $(p,p,p)$, which allows us to apply the modular method to each coupled system of equations. The modular method is practical until the levels of all three equations in each system are too large.

In addition, we use linear forms in logarithms to bound the exponent $p$. After reducing to a finite list of the exponents, we then apply local methods to each exceptional exponent $p$. For the remaining small prime exponents, we use the Thue equation solver in {\tt Magma}. 

Additionally, use \cite{bennett} and \cite{Darmon-Merel} to treat equations with evident solutions.

We conclude by proving the following theorem:
\begin{theorem}
    For $2 \le k \le 42$ even, the equation
    \[
        S_k(x)=1^k+2^k+\cdots+x^k=y^n
    \]
    has no solutions $(x,y,n)\in\N^3$ with $x,y,n\ge2$.
\end{theorem}

\section{Descent step}

For even $k$, we may write 
\begin{equation}
    S_k(x)=\frac{1}{C_{k}}x(x+1)(2x+1)T_{k}(x)   
    \label{2.2}
\end{equation}
where $C_k\in\Z$. Let $D=\gcd (x(x+1),T_{k}(x))$ and $d=\gcd (2x+1,T_{k}(x))$ where we consider the gcd over $\Z$ for a particular $x\in\Z$. For a fixed $k$, we will use the following result to get a finite list of possible values for $D$ and $d$.

\begin{lemma}
    \label{resultant-lemma}
    Let $f(x)=\alpha x+\beta, g(x)=\gamma x^n+\cdots\in\Z[x]$ such that $\alpha,\gamma\neq0$ and $|\Res(f,g,x)|=\delta$ is non-zero. Then $\forall x_0\in\Z$, 
    \begin{equation*}
        \gcd(f(x_0),g(x_0))\mid\delta.
    \end{equation*}
    Furthermore, let $p$ be a prime such that $p\mathrel\|\delta$, $p\nmid\alpha$, and $p\mid\alpha^{n-1}g'\left(-\frac{\beta}{\alpha}\right)$. If $p\mid\gcd(f(x_0),g(x_0))$, then $p\mathrel\| g(x_0)$.
\end{lemma}

\begin{proof}
    Recall that there are polynomials $A(x),B(x)\in\Z[x]$ such that
    \begin{equation}
        \label{resultant-equation}
        A(x)f(x)+B(x)g(x)=\Res(f,g,x)=\alpha^ng\left(-\frac{\beta}{\alpha}\right).
    \end{equation}
    Let $x_0\in\Z$. Then $\gcd(f(x_0),g(x_0))\mid\delta$ by \eqref{resultant-equation}. Suppose a prime $p\mid\gcd(f(x_0),g(x_0))$. Consider $\alpha^ng\left(\frac{y-\beta}{\alpha}\right)$ as a polynomial of degree $n$ in $\Z[y]$ and write
    \begin{equation*}
        \alpha^ng\left(\frac{y-\beta}{\alpha}\right) = \alpha^ng\left(-\frac{\beta}{\alpha}\right) + \alpha^{n-1}g'\left(-\frac{\beta}{\alpha}\right)y+y^2p(y)
    \end{equation*}
    for some $p(y)\in\Z[y]$. Then for any $x_1\in\Z$,
    \begin{equation*}
        \alpha^ng\left(\frac{px_1-\beta}{\alpha}\right)\equiv\alpha^ng\left(-\frac{\beta}{\alpha}\right)+\alpha^{n-1}g'\left(-\frac{\beta}{\alpha}\right)(px_1)\equiv\alpha^ng\left(-\frac{\beta}{\alpha}\right)=\Res(f,g,x)\mod{p^2},
    \end{equation*}
    where the last congruence is due to $p\mid\alpha^{n-1}g'\left(-\frac{\beta}{\alpha}\right)$. Therefore, $p^2\nmid\alpha^ng\left(\frac{px_1-\beta}{\alpha}\right)$ because $p\mathrel\|\Res(f,g,x)$. Since $p\mid f(x_0)=\alpha x_0+\beta$, write $x_0=\frac{px_1-\beta}{\alpha}$ for some $x_1\in\Z$. In particular, $p\mathrel\|\alpha^ng\left(\frac{px_1-\beta}{\alpha}\right)=\alpha^ng(x_0)$ and thus $p\mathrel\| g(x_0)$.
\end{proof}

For any even value of $k$, we may write
\begin{equation*}
    S_k(x)=\frac{1}{C_k}x(x+1)(2x+1)T_k(x),
\end{equation*}
where $C_k\in\Z$ and $T_k(x)\in\Z[x]$. For a fixed even $k$, let $D_0=\gcd(x,T_k(x))$, $D_1=\gcd(x+1,T_k(x))$ and $D_2=\gcd(2x+1,T_k(x))$. Using the lemma, we get a finite list of possible values for $D_0,D_1$, and $D_2$.

The values for $D_0,D_1,$ and $D_2$, along with $C_k$, are listed in Table~\ref{dtable}.
{\tiny
\begin{center}
\begin{table}[ht]
\label{dtable}
    \begin{tabular}{|c|c|c|c|}
        \hline
        $k$ & $D_0$ and $D_1$ divides & $D_2$ divides & $C_k$ \\
        \hline
        2 & 1 &  1 &  $2 \cdot 3$  \\
        
        4 & 1 & 7 & $2 \cdot 3 \cdot 5$ \\
        
        6 & 1 & 31 & $ 2\cdot 3 \cdot 7$ \\
        
        8 & 3 & $3\cdot127$ & $2 \cdot 3^2 \cdot 5$ \\

        10 & 5 & $5 \cdot 7 \cdot 73$ & $2 \cdot 3 \cdot 11$ \\

        12 & 691 & $23 \cdot 89 \cdot 691$ & $2 \cdot 3 \cdot 5 \cdot 7 \cdot 13$ \\
        
        14 & $3\cdot5\cdot7$ & $3\cdot5\cdot7\cdot8191$ & $2 \cdot 3^2 \cdot 5$ \\

        16 & 3617 & $7\cdot31\cdot151\cdot3617$ & $2\cdot3\cdot5\cdot17$ \\
        
        18 & $5\cdot43867$ & $5\cdot43867\cdot131071$ & $2\cdot3\cdot5\cdot7\cdot19$ \\

        20 & $3\cdot7\cdot283\cdot617$ & $3\cdot7\cdot283\cdot617\cdot524287$ & $2\cdot3^2\cdot5\cdot7\cdot11$ \\

        22 & $5\cdot11\cdot131\cdot593$ & $5\cdot7^2\cdot11\cdot127\cdot131\cdot337\cdot593$ & $2\cdot3\cdot5\cdot23$ \\

        24 & $5\cdot103\cdot2294797$ & $5\cdot47\cdot103\cdot178481\cdot2294797$ & $2\cdot3\cdot5^2\cdot7\cdot13$ \\

        26 & $3^2\cdot7\cdot13\cdot657931$ & $3^2\cdot7\cdot13\cdot31\cdot601\cdot1801\cdot657931$ & $2\cdot3^3\cdot7$\\

        28 & $7\cdot9349\cdot362903$ & $7^2\cdot73\cdot9349\cdot262657\cdot362903$ & $2\cdot3\cdot5\cdot29$ \\

        30 & $5\cdot1721\cdot1001259881$ & $5\cdot233\cdot1103\cdot1721\cdot2089\cdot1001259881$ & $2\cdot3\cdot7\cdot11\cdot31$ \\

        32 & $3\cdot7\cdot11\cdot37\cdot683\cdot305065927$ & $3\cdot7\cdot11\cdot37\cdot683\cdot305065927\cdot2147483647$ & $2\cdot3^2\cdot5\cdot7\cdot11\cdot17$\\

        34 & $5\cdot7\cdot17\cdot151628697551$ & $5\cdot7^2\cdot17\cdot23\cdot89\cdot599479\cdot151628697551$ & $2\cdot3\cdot5\cdot7$\\

        36 & $26315271553053477373$ & $31\cdot71\cdot127\cdot122921\cdot26315271553053477373$ & $2\cdot3\cdot5\cdot7\cdot13\cdot19\cdot37$ \\

        38 & $3\cdot5\cdot7\cdot13\cdot19\cdot154210205991661$ & $3\cdot5\cdot7\cdot13\cdot19\cdot223\cdot616318177\cdot154210205991661$ & $2\cdot3^2\cdot5\cdot7\cdot13$ \\

        40 & $7\cdot137616929\cdot1897170067619$ & $7^2\cdot79\cdot8191\cdot121369\cdot137616929\cdot1897170067619$ & $2\cdot3\cdot5\cdot7\cdot11\cdot41$ \\

        42 & $5 \cdot 11 \cdot 1520097643918070802691$ & $5 \cdot 11 \cdot 13367 \cdot 164511353 \cdot 1520097643918070802691$ & $2 \cdot 3 \cdot 5 \cdot 7 \cdot 11 \cdot 43$ \\

        \hline
    \end{tabular}
    \medskip
\caption{Summary of constants}
\end{table}
\end{center}
}

\begin{lemma}
    \label{resultant-valuation}
    Let $f(x)=\alpha x+\beta$ and $g(x)=\gamma x^n+\cdots\in\Z[x]$ such that $\alpha,\gamma\neq0$ and $|\Res(f,g,x)|=\delta$ is non-zero. Further, let $x_0\in\Z$ and $p$ be a prime such that $p\nmid\alpha$ and $v_p(f(x_0))>v_p(g(x_0))$. Then
    \[
        v_p(g(x_0)) = v_p(\delta).
    \]
    Here $p^{v_p(n)}\| n$.
\end{lemma}

\begin{proof}
    Write $G(y) := \alpha^ng\left(\frac{y-\beta}{\alpha}\right)\in\Z[y]$. First, note that
    \[
        |G(0)| = \left|\alpha^ng\left(\frac{0-\beta}{\alpha}\right)\right| =\left|\alpha^ng\left(-\frac{\beta}{\alpha}\right)\right|=|\delta|
    \]
    so that $v_p(G(0))=v_p(\delta)$. Further,
    \[
        G(f(x_0)) = \alpha^ng\left(\frac{\alpha x_0+\beta-\beta}{\alpha}\right) =\alpha^ng(x_0).
    \]
    Since $p\nmid\alpha$, we get
    \[
        v_p(G(f(x_0))=v_p(\alpha^ng(x_0))=v_p(g(x_0)).
    \]
    Next, we have that
    $y\mid G(y)-G(0)$ in $\Z[y]$. In particular, for $x_0\in\Z$, we have $f(x_0)\mid G(f(x_0))-G(0)$ and thus
    \[
        v_p(G(f(x_0))-G(0)) \ge v_p(f(x_0)).
    \]
    Hence,
    \[
        v_p(G(f(x_0))-G(0))\ge v_p(f(x_0))>v_p(g(x_0))=v_p(G(f(x_0))).
    \]
    Since $G(0) = G(f(x_0)) - (G(f(x_0))-G(0))$, by the ultrametric property we get
    \[
        v_p(\delta)=v_p(G(0)) = \min\{v_p(G(f(x_0))), v_p(G(f(x_0))-G(0))\}=v_p(G(f(x_0)))=v_p(g(x_0)).
    \]
\end{proof}

We'll also need the following lemma.

\begin{lemma}
    \label{coprime-rational-split}
    Suppose 
    \begin{equation*}
        a_0\cdots a_k=\frac{cy^n}{d}
    \end{equation*}
    for $a_i,c,d,y,n\in\N$, where $\gcd(a_i,a_j)=\gcd(c,d)=1$ for $i\neq j$. Then $\exists y_i,c_i,d_i\in\N$ such that
    \begin{align*}
        &c=c_0\cdots c_k,\quad d=d_{-1}d_0\cdots d_k,\quad y=y_0\cdots y_k ,\\
        &\gcd(d_i,d_j) = \gcd(c_i,c_j)=\gcd(y_i,y_j)=1\text{ for } i\neq j ,\\
        &d_i = \prod_{p\mid a_i}p^{v_p(d)} ,\\
        &d_{-1} = \prod_{p\nmid a_0\cdots a_k}p^{v_p(d)} ,\\
        &c_i = \gcd(c,a_i) ,\\
        &a_i = \frac{c_iy_i^n}{d_i}.
    \end{align*}
\end{lemma}

\begin{proof}
    We first prove the case when $d = 1$. Since $c\mid a_0\cdots a_k$ so that $\gcd(c,a_0\cdots a_k)=c$. Along with the assumption that $\gcd(a_i,a_j)=1$, we have
    $\gcd(c,a_0)\cdots\gcd(c,a_k)=\gcd(c,a_0\cdots a_k)=c$ and $\gcd(c,a_i)$ are pairwise coprime. Write $c_i=\gcd(c,a_i)$. Then
    \begin{equation*}
        \left(\frac{a_0}{c_0}\right)\cdots\left(\frac{a_k}{c_k}\right)=y^n.
    \end{equation*}
    Since $\gcd\left(\frac{a_i}{c_i},\frac{a_j}{c_j}\right)=1$ for $i\neq j$, we have that for any prime $p$,  $v_p\left(\left(\frac{a_0}{c_0}\right)\cdots\left(\frac{a_k}{c_k}\right)\right) = nv_p(y)$ so $n\mid v_p\left(\left(\frac{a_0}{c_0}\right)\cdots\left(\frac{a_k}{c_k}\right)\right)$. Thus, $v_p\left(\frac{a_i}{c_i}\right) \equiv 0 \pmod n$, and by unique factorization we have
    \[  
        \frac{a_i}{c_i}=y_i^n
    \]
    for some $y_i\in\N$.

    Now we treat the case of general $d$. Define $d_i=\prod_{p\mid a_i} p^{v_p(d)}$ for $i=0,...,k$ and $d_{-1} = \prod_{p\nmid a_0\cdots a_k} p^{v_p(d)}$ so that $d_i$ are pairwise coprime. Then we have
    \begin{equation}
        (d_0a_0)\cdots(d_ka_k)d_{-1}=cy^n
    \end{equation}
    and $d_{-1},d_ia_i$ are pairwise coprime since $d_i$ and $a_i$ are. By the case when $d = 1$, there exist $y_i,c_i\in\N$ such that $c=c_0 \cdots c_k$, $y = y_1 \cdots y_k$, $\gcd(c_i,c_j)=\gcd(y_i,y_j)=1$ for $i \not= j$, 
    \[
        d_i a_i = c_i y_i ^n, \quad c_{-1}=\gcd(c,d_{-1}),\quad\text{and}\quad c_i=\gcd(c,d_ia_i) \text{ for $i=0,...,k$}.
    \]
    Since $\gcd(c,d)=1$ and $d_{-1}\mid d$, we have $c_{-1}=\gcd(c,d_{-1})=1$. Similarly, $\gcd(d,c_ia_i)=\gcd(d,a_i)$.
\end{proof}

Fix $k\in\N$ even and write
\begin{align*}
    S_k(x)=\frac{1}{C_k}x(x+1)(2x+1)T_k(x).
\end{align*} 
Let $\delta_0=|\Res(x,T_k(x),x)|,\delta_1=|\Res(x+1,T_k(x),x)|$, and $\delta_2=|\Res(2x+1,T_k(x),x)|$. Let $p$ be a prime and consider the following hypothesis:
\begin{align}
    &p\mid\gcd(x,T_k(x)), \text{ }p\mathrel\|\delta_0\text{ and } \delta_0\mid T_k'(0), \label{delta0} \\
    &p\mid\gcd(x+1,T_k(x)),\text{ }p\mathrel\|\delta_1, \text{ and } \delta_1\mid T_k'(-1), \text{ and} \label{delta1} \\
    &p\mid\gcd(2x+1,T_k(x)), \text{ }p\neq2, \text{ }p\mathrel\|\delta_2, \text{ and } \delta_2\mid 2^{\deg T_k-1}T_k' (-1/2). \label{delta2}
\end{align}
In view of Lemma \ref{resultant-lemma}, if any of \eqref{delta0}, \eqref{delta1}, or \eqref{delta2} are satisfied, then  $p\mathrel\|T_k(x)$.

\begin{prop}
\label{descent}
    For even $k$, suppose
    \begin{equation}
        \label{main-equation}
        S_k(x) = \frac{1}{C_{k}}x(x+1)(2x+1)T_{k}(x) = y^n
    \end{equation}
    has solution $x,y,n$ in $\N$ with $n \ge 2$. Write $f_0(x)=x, f_1(x)=x+1$, $f_2(x)=2x+1$. For $i=0,1,2$, let $D_i = \gcd(f_i(x),T_k(x))$, so that the $D_i$ are pairwise coprime and thus $\frac{T_k(x)}{D_0D_1D_2}\in\Z$. Then for some pairwise coprime $\alpha_i,d_j \in \N$, $0\le i\le 3$,$-1\le j\le 3$ satisfying
    \begin{equation}
        \alpha_0\alpha_1\alpha_2\alpha_3=\frac{C_k}{\gcd(D_0^2D_1^2D_2^2,C_k)},\quad d_{-1}d_0d_1d_2d_3=\frac{D_0^2D_1^2D_2^2}{\gcd(D_0^2D_1^2D_2^2,C_k)}, \quad
            d_i\mid\frac{D_i^2}{\gcd(D_i^2,C_k)}
    \end{equation}
    for $i=0,1,2$, 
    \begin{equation*}
        d_0\frac{x}{D_0}, \quad d_1\frac{x+1}{D_1}, \quad d_2\frac{2x+1}{D_2}, \quad d_3\frac{T_k(x)}{D_0D_1D_2}
    \end{equation*}
    are pairwise coprime and the system of equations 
    \begin{align}
        \alpha_1D_1\frac{u_1^n}{d_1}-\alpha_0D_0\frac{u_0^n}{d_0} &= 1 \label{14}\\
        \alpha_2D_2\frac{u_2^n}{d_2}-2\alpha_0D_0\frac{u_0^n}{d_0} &= 1 \label{15} \\
        2\alpha_1D_1\frac{u_1^n}{d_1} - \alpha_2D_2\frac{u_2^n}{d_2} &= 1 \label{16} \\
        \frac{T_k(x)}{D_0D_1D_2} &= \alpha_3\frac{u_3^n}{d_3} \label{17}
    \end{align}
    is solvable in $u_0, u_1, u_2 \in \N$. Further, $f_i(x) = \alpha_iD_i\frac{u_i^n}{d_i}$,
    \[
        v_2\left(\frac{\alpha_0D_0}{d_0}\frac{\alpha_1D_1}{d_1}\right)\equiv1\mod n, \text{ and } v_3\left(\frac{\alpha_0D_0}{d_0}\frac{\alpha_1D_1}{d_1}\frac{\alpha_2D_2}{d_2}\right) \equiv1\mod n.
    \]
    Finally, suppose $n\ge3$. Let $\delta_i=|\Res(f_i(x),T_k(x),x)|$. If $p\mid d_i$, then $v_p(D_i)=v_p(\delta_i)$. Instead, let $p\mid D_i$ be a prime. 
    \begin{enumerate}
        \item If \eqref{delta0} holds, then $v_p(d_0)=v_p\left(\frac{D_0^2}{\gcd(D_0^2,C_k)}\right)$ and $p\nmid\alpha_3$. \label{7.1}
        \item If \eqref{delta1} holds, then $v_p(d_1)=v_p\left(\frac{D_1^2}{\gcd(D_1^2,C_k)}\right)$ and $p\nmid\alpha_3$, \label{7.2}
        \item If \eqref{delta2} holds, then then $v_p(d_2)=v_p\left(\frac{D_2^2}{\gcd(D_2^2,C_k)}\right)$ and $p\nmid\alpha_3$. \label{7.3}
    \end{enumerate}
\end{prop}

\begin{proof}
    We first note that $D_0,D_1$, and $D_2$ are pairwise coprime since $x$, $x+1$, and $2x+1$ are. Multiplying \eqref{main-equation} by $\frac{C_k}{D_0^2D_1^2D_2^2}$, we have
    \begin{equation}
        \label{main-equation-divided}
        \frac{x}{D_0}\frac{x+1}{D_1}\frac{2x+1}{D_2}\frac{T_{k}(x)}{D_0D_1D_2} = \frac{C_ky^n}{D_0^2D_1^2D_2^2} =  \frac{C_k}{\gcd(D_0^2D_1^2D_2^2,C_k)}\frac{y^n}{\frac{D_0^2D_1^2D_2^2}{\gcd(D_0^2D_1^2D_2^2,C_k)}}.
    \end{equation}
    Since $\gcd\left(\frac{C_k}{\gcd(D_0^2D_1^2D_2^2,C_k)}, \frac{D_0^2D_1^2D_2^2}{\gcd(D_0^2D_1^2D_2^2,C_k)}\right)=1$ and $\frac{x}{D_0},\frac{x+1}{D_1},\frac{2x+1}{D_2}$, and $\frac{T_k(x)}{D_0D_1D_2}$ are pairwise coprime, by Lemma \ref{coprime-rational-split} there exist pairwise coprime $\alpha_0,\alpha_1,\alpha_2,\alpha_3,d_{-1},d_0,d_1,d_2,d_3\in\N$ such that
    \begin{equation*}
        d_0\frac{x}{D_0},\quad d_1\frac{x+1}{D_1},\quad d_2\frac{2x+1}{D_2}, \quad d_3\frac{T_k(x)}{D_0D_1D_2}
    \end{equation*}
    are pairwise coprime,
    \begin{equation*}
        \frac{x}{D_0}=\alpha_0\frac{u_0^n}{d_0}, \quad \frac{x+1}{D_1}=\alpha_1\frac{u_1^n}{d_1},\quad \frac{2x+1}{D_2}=\alpha_2\frac{u_2^n}{d_2},\quad\frac{T_k(x)}{D_0D_1D_2}=\alpha_3 \frac{u_3^n}{d_3},
    \end{equation*}
    such that
    \begin{equation*}
        \alpha_0\alpha_1\alpha_2\alpha_3=\frac{C_k}{\gcd(D_0^2D_1^2D_2^2,C_k)},\quad d_{-1}d_0d_1d_2d_3=\frac{D_0^2D_1^2D_2^2}{\gcd(D_0^2D_1^2D_2^2,C_k)},
    \end{equation*}
    and 
    \begin{equation*}
        d_0 = \prod_{p\mid \frac{x}{D_0}}p^{v_p},\quad d_1 = \prod_{p\mid \frac{x+1}{D_1}}p^{v_p},\quad d_2 = \prod_{p\mid \frac{2x+1}{D_2}}p^{v_p}, \quad d_3 = \prod_{p\mid \frac{T_k(x)}{D_0D_1D_2}}p^{v_p},
    \end{equation*}
    and 
    \begin{equation*}
        d_{-1}=\prod_{p\nmid \frac{x(x+1)(2x+1)T_k(x)}{D_0^2D_1^2D_2^2}}p^{v_p},
    \end{equation*}
    where $v_p=v_p\left(\frac{D_0^2D_1^2D_2^2}{\gcd(D_0^2D_1^2D_2^2,C_k)}\right)$. As a consequence of \cite[Lemma 3.1 and Lemma 3.2]{Hajdu}, we have that $v_2(x(x+1)) \equiv v_3(x(x+1)(2x+1))\equiv1\mod n$. In particular,
    \[
        1\equiv v_2(x(x+1)) \equiv v_2\left(\frac{\alpha_0D_0u_0^n}{d_0}\frac{\alpha_1D_1u_1^n}{d_1}\right) \equiv v_2\left(\frac{\alpha_0D_0}{d_0}\frac{\alpha_1D_1}{d_1}\right) \mod n
    \]
    and similarly $v_3\left(\frac{\alpha_0D_0}{d_0}\frac{\alpha_1D_1}{d_1}\frac{\alpha_2D_2}{d_2}\right) \equiv1\mod n.$
    
    We can substitute and rearrange the expressions for $x,x+1$, and $2x+1$ to get the following three equations
    \begin{align*}
        \alpha_1D_1\frac{u_1^n}{d_1}-\alpha_0D_0\frac{u_0^n}{d_0} &= 1,\\
        \alpha_2D_2\frac{u_2^n}{d_2}-2\alpha_0D_0\frac{u_0^n}{d_0} &= 1, \\
        2\alpha_1D_1\frac{u_1^n}{d_1} - \alpha_2D_2\frac{u_2^n}{d_2} &= 1.
    \end{align*}

    Fix $i\in\{0,1,2\}$. Suppose $p\mid d_i$ so that $p\mid\frac{f_i(x)}{D_i}$ and $v_p>0$. Then $v_p(f_i(x))>v_p(D_i)$ and hence $v_p(f_i(x))>v_p(T_k(x))$. Thus,
    \[
        v_p(D_i)=\min\{v_p(f_i(x)),v_p(T_k(x))\} = v_p(T_k(x))<v_p(f_i(x)).
    \]
    By Lemma \ref{resultant-valuation}, we have $v_p(D_i)=v_p(T_k(x))=v_p(\delta_i)$. Next, since $D_0,D_1,D_2$ are pairwise coprime, $\gcd(D_0^2D_1^2D_2^2,C_k)=\gcd(D_0^2,C_k)\gcd(D_1^2,C_k)\gcd(D_2^2,C_k)$. Then 
    \begin{equation}
        v_p=v_p\left(\frac{D_0^2D_1^2D_2^2}{\gcd(D_0^2D_1^2D_2^2,C_k)}\right) = v_p\left(\frac{D_0^2}{\gcd(D_0^2,C_k)}\right) +v_p\left(\frac{D_1^2}{\gcd(D_1^2,C_k)}\right) +v_p\left(\frac{D_2^2}{\gcd(D_2^2,C_k)}\right).
    \end{equation}
    Since $p\mid f_i(x)$ and $D_0,D_1,D_2$ are pairwise coprime, $p\nmid D_j$ for $j\neq i$ and in particular,
    \begin{equation*}
        v_p = v_p\left(\frac{D_i^2}{\gcd(D_i^2,C_k)}\right).
    \end{equation*}
    Thus
    \begin{equation*}
        d_i\mid\frac{D_i^2}{\gcd(D_i^2,C_k)}.
    \end{equation*}
    
    For the remainder of the proof, suppose that condition \eqref{7.1} is satisfied, that is for a prime $p\mid D_0$, we have $\delta_0\neq0$, $p\mathrel\|\delta_0$, and $\delta_0\mid T_k'(0)$. Then by Lemma \ref{resultant-lemma},  $p\mathrel\|T_k(x)$. Since 
    \[
        \frac{T_k(x)}{D_0D_1D_2}=\alpha_3\frac{u_3^n}{d_3},
    \]
    $\gcd(\alpha_3,d_3)=1$, and $\frac{T_k(x)}{D_0D_1D_2}\in\Z$, we have that $\alpha_3\mid\frac{T_k(x)}{D_0D_1D_2}$. Thus $p\nmid\alpha_3$. Additionally, if $p\mid\frac{D_0^2}{\gcd(D_0^2,C_k)}$, then $p\mid D_0$ so that $p\nmid\frac{T_k(x)}{D_0D_1D_2}$ and thus $p\nmid d_3$.

    Let $p\mid d_{-1}$ so that  $p\mid\frac{D_0^2D_1^2D_2^2}{\gcd(D_0^2D_1^2D_2^2,C_k)}$ and  $p\nmid\frac{x(x+1)(2x+1)T_k(x)}{D_0^2D_1^2D_2^2}$. Since $D_0,D_1,D_2$ are pairwise coprime, without loss of generality assume that $p\mid\frac{D_0^2}{\gcd(D_0^2,C_k)}$. In particular, we have $p\nmid\frac{x}{D_0}$ and thus 
    \begin{equation*}
        v_p(x)=v_p(D_0),\quad v_p(D_1)=0=v_p(D_2).
    \end{equation*}
    As $p\mathrel\|T_k(x)$, we have $v_p(D_0)=1$ so that $v_p(D_0^2D_1^2D_2^2)=2$. Thus by equation \eqref{main-equation-divided},
    \begin{equation*}
        0 = v_p\left(\frac{x(x+1)(2x+1)T_k(x)}{D_0^2D_1^2D_2^2}\right) =  v_p\left(\frac{C_k}{\gcd(D_0^2D_1^2D_2^2,C_k)}\right) + nv_p(y) - v_p\left(\frac{D_0^2D_1^2D_2^2}{\gcd(D_0^2D_1^2D_2^2,C_k)}\right).
    \end{equation*}
    If $v_p(C_k)=0$, then we get $0 = 0 + nv_p(y)-2$ so that $nv_p(y)=2$, which is a contradiction when $n\ge3$. If $v_p(C_k)=1$, then we get $0 = nv_p(y)-1=0$ so that $nv_p(y)=1$, which again is a contradiction when $n\ge 2$. Finally, if $v_p(C_k)\ge2$ then
    \begin{equation}
        v_p\left(\frac{D_0^2D_1^2D_2^2}{\gcd(D_0^2D_1^2D_2^2,C_k)}\right)=0
    \end{equation}
    which is a contradiction. Thus $p\nmid d_{-1}$. We also have that if $p\mid\frac{D_0^2}{\gcd(D_0^2,C_k)}$, then $p\mid x$ and so $p\nmid x+1,2x+1$. Hence $p\nmid d_1,d_2$. Since $d_{-1}d_0d_1d_2d_3=\frac{D_0^2D_1^2D_2^2}{\gcd(D_0^2D_1^2D_2^2,C_k)}$, we thus have
    \begin{equation}
        v_p(d_0)=v_p\left(\frac{D_0^2}{\gcd(D_0^2,C_k)}\right).
    \end{equation}
    The cases when condition \eqref{7.2} or \eqref{7.3} are satisfied follow similarly.
\end{proof}

\begin{remark}
    \label{descent-symmetry}
    The contrapositive of Proposition \ref{descent} is if $\forall\alpha_i,d_i\in\N$, $i=0,1,2,3$, satisfying the conditions outlined in the proposition, the system of equations \eqref{14}-\eqref{16} is insolvable in $u_0,u_1,u_2\in\N$, then there are no solutions to \eqref{main} for a fixed $k$. Suppose $\alpha_1>\alpha_0$ and that $(X,Y,Z)=(u_0,u_1,u_2)$ solves the system
    \begin{align*}
        \alpha_1D_1\frac{Y^n}{d_1}-\alpha_0D_0\frac{X^n}{d_0} &= 1\\
        \alpha_2D_2\frac{Z^n}{d_2}-2\alpha_0D_0\frac{X^n}{d_0} &= 1 \\
        2\alpha_1D_1\frac{Y^n}{d_1} - \alpha_2D_2\frac{Z^n}{d_2} &= 1
    \end{align*}
    where $\frac{x}{D_0}=\alpha_0\frac{u_0^n}{d_0}, \frac{x+1}{D_1}=\alpha_1\frac{u_1^n}{d_1},$ and $\frac{2x+1}{D_2}=\alpha_2\frac{u_2^n}{d_2}$. Call this system (a). Note $u_0,u_1,u_2>0$. Then (if $n$ is odd) $(X,Y,Z)=(-u_1,-u_0,-u_2)$ solves the system
    \begin{align*}
        \alpha_0D_0\frac{X^n}{d_0}-\alpha_1D_1\frac{Y^n}{d_1} &= 1\\
        \alpha_2D_2\frac{Z^n}{d_2}-2\alpha_1D_1\frac{Y^n}{d_1} &= 1 \\
        2\alpha_0D_0\frac{X^n}{d_0} - \alpha_2D_2\frac{Z^n}{d_2} &= 1
    \end{align*}
    Call this system (b). Consider the set $S$ of all possible systems descending from \eqref{main}. Then system (a) is in $S$ if and only if system (b) is in $S$. Thus, instead of showing that each system in $S$ has no solutions in $\N$, we can show that each system in $S$ such that $\alpha_1>\alpha_0$ has no solutions $(u_0,u_1,u_2)$ in $\Z$, such that $u_0,u_1,u_2$ have the same sign and $\frac{x}{D_0}=\alpha_0\frac{u_0^n}{d_0}$ if $u_0>0$ or $\frac{x}{D_1}=-\alpha_1\frac{u_1^n}{d_1}$ if $u_0<0$.
\end{remark}

\section{The modular method}

Without loss of generality, it suffices to consider the case of $n = p$ prime. Using Proposition \ref{descent}, solutions to \eqref{main} give rise to solutions to \eqref{14}--\eqref{16} which are of generalized Fermat type with one of the variables set to $1$.

We now briefly summarize the modular method as it is applied to resolve generalized Fermat type equations of the form
\begin{equation}
\label{fermat-equ}
  A x^p + B y^p + C z^p = 0,
\end{equation}
where $A=\frac{p_1}{q_1},B=\frac{p_2}{q_2},C=\frac{p_3}{q_3}\in\Q$ are non-zero with $p_1,p_2,p_3,q_1,q_2,q_3$ pairwise coprime which are not $p$th powers and such that $p \nmid ABC$ and $Ax^p,By^p,Cz^p\in\Z$. Our treatment follows \cite{houle2025}.

A solution $(a,b,c) \in \Z^3$ of \eqref{fermat-equ} is primitive if $a,b,c$ are mutually coprime. Up to permutation of the terms and negation in \eqref{fermat-equ}, we can assume that $By^p$ is even and $A x^p \equiv -1 \pmod 4$.

To a primitive non-trivial solution $(a,b,c)$ of \eqref{fermat-equ} we attach the Frey elliptic curve
\begin{equation}
    Y^2 = X (X -A a^p) (X + B b^p).
\end{equation}

Consider the Galois representation
\[ \rhobar_{E,p} : G_\Q \rightarrow \GL_2(\F_p) \] attached to the elliptic curve $E$ over $\Q$, which is obtained by letting $G_\Q := \Gal(\overline{\Q}/\Q)$ act linearly on $E[p]$. As the $E$ is a Frey elliptic curve over $\Q$, we have by modularity and level lowering that
\begin{equation}
\label{modp-isom}
    \rhobar_{E,p} \simeq \rhobar_{f,\Fp}
\end{equation}
where $f$ is a classical newform  of weight $2$, level $N = N(\bar\rho_{E,p})$ equal to the Serre level of $\rhobar_{E,p}$, and trivial character and $\Fp$ is a prime of the coefficient field $K$ of $f$ above $p$. 

By \cite{houle2025}, we can compute the Serre level $N$ of $\rhobar_{E,p}$. Let $R = ABC$. If $v_2(ABC)\ge0$, then
\begin{equation}
    \label{kraus-N}
    N(\bar\rho_{E,p}) = \case{
        2\rad'(R) & \text{if $v_2(ABC)=0$ or $v_2(ABC)\ge5$}, \\
        \rad'(R) & \text{if $v_2(ABC)=4$}, \\
        2^5\rad'(R) & \text{if $v_2(ABC)=1$ and $b$ is odd}, \\
        2^3\rad'(R) & \text{if $v_2(ABC)=2,3$ and $b$ is odd}, \\
        2\rad'(R) & \text{if $1\le v_2(ABC)\le3$ and $b$ is even},
    }
\end{equation}
where $\rad'(n)$ denotes the product of all odd primes dividing $n$.
If instead $v_2(ABC)<0$, then
\begin{equation}
    \label{james-N}
    N(\bar\rho_{E,p}) = \case{
        \rad'(R) & \text{if $v_2(ABC)=4-p$}, \\
        2^5\rad'(R) & \text{if $v_2(abc)=1$ and $v_2(ABC)=1-p$,} \\
        2^3\rad'(R) & \text{if $v_2(abc)=1$ and $v_2(ABC)=2-p,3-p$,} \\
        2\rad'(R) & \text{if $v_2(abc)\ge2$ and $v_2(ABC)\neq 4-p$,} \\
        2\rad'(R) & \text{if $v_2(abc)=1$, $p\ge 7$, and $v_2(ABC)\ge 5-p$.}
    }
\end{equation}
    
The conductor $N_E$ of $E/\Q$ is square-free except for powers of $2$. Hence, the Frey elliptic curve has either good or multiplicative bad reduction modulo primes $\ell \not=2$.

Computing the finite set of newforms $f$ of weight $2$, level $N$, and trivial character we now consider the isomorphism in \eqref{modp-isom} for each possible $f$. The idea behind the modular method is to show this is not possible for each $f$, i.e.\  eliminate every newform $f$ as a possibility in the isomorphism \eqref{modp-isom} for $p$ sufficiently large.

Let $\ell \nmid Np$ be a prime so that $\rhobar_{E,p}$ is unramified at $\ell$. Let $\Frob_\ell \in G_\Q := \Gal(\overline{\Q}/\Q)$ be a Frobenius element at $\ell$. We may thus consider the well defined trace value 
\begin{equation}
    a_\ell(f) = \tr \rho_{f,\Fp}(\Frob_\ell) \in \calO_{K},
\end{equation}
where $\rho_{f,\Fp}$ is the $\Fp$-adic representation attached to the newform $f$.

If $\ell \nmid N_E$, where $N_E$ is the conductor of $E/\Q$, then $E$ has good reduction at $\ell$ and we may consider
\begin{equation}
    a_\ell(E) = \tr \rho_{E,p}(\Frob_\ell) \in \Z,
\end{equation}
where $\rho_{E,p}$ is the $p$-adic representation attached to $E$.

The isomorphism \eqref{modp-isom} implies that
\begin{equation}
    a_\ell(E) \equiv a_\ell(f) \pmod \Fp,
\end{equation}
and hence
\begin{equation}
\label{good-reduc}
    \Norm_{K/\Q}(a_\ell(E)-a_\ell(f)) \equiv 0 \pmod p
\end{equation}
where $\Norm_{K/\Q}(\cdot)$ is the norm from $K$ to $\Q$. If the above norm is non-zero, we obtain a bound on the prime exponent $p$, i.e.\ for $p$ sufficiently large, \eqref{modp-isom} does not hold for the newform $f$.

If $\ell \mid N_E$, then $E$ has multiplicative bad reduction at $\ell$. It can be shown that $\text{WD}_\ell(\left\{\rho_{E,p}\right\})$ is a special representation with character being an unramified character of order dividing $2$. It follows that
\begin{equation}
\pm (\ell+1) \equiv a_\ell(f) \pmod \Fp,
\end{equation}
and hence
\begin{equation}
\label{bad-reduc}
    \Norm_{K/\Q}((\ell+1)^2-a_\ell(f)^2) \equiv 0 \pmod p,
\end{equation}
which gives a bound on the prime exponent $p$ as the above norm is non-zero by the Deligne-Ramanujan bounds, i.e.\ for $p$ sufficiently large, \eqref{modp-isom} does not hold for the newform $f$.

Let $S_2(N)$ denote the space of weight $2$ cusp forms of level $N$. In what follows, when we write $f \in S_2(N)$ we restrict to $f$ being a newform.
\begin{lemma}
\label{p-values}
    Let $(a,b,p)$ be a solution to 
    \begin{equation}
    \label{GFE}
        Ax^n-By^n=1
    \end{equation}
    where $A,B\in\Q$ are non-zero with pairwise coprime numerators and denominators such that either $v_2(A)>0$ or $v_2(B)>0$ and $Aa^p,Bb^p\in\Z$. Let $\mathcal N$ be the set of possible levels given by \eqref{kraus-N}. Then if $\ell$ is a prime number such that $\ell\nmid N$ for any $N\in\mathcal N$,
    \begin{equation*}
        p\in\bigcup_{N\in\mathcal N}\bigcup_{f\in S_2(N)}\bigcup_{w\in W_\ell} \left\{p \text{ prime} : p\mid B_\ell(f,w)\text{ or } p=\ell\right\},
    \end{equation*}
    where 
    \[
        W_\ell=\left\{(w_0,w_1)\in(\Z/\ell\Z)^2: Aw_0^p-Bw_1^p\equiv1\mod \ell\right\},
    \]
    and 
    \begin{equation*}
        B_\ell(f,w)= \case{
            \Norm_{K_f/\Q}(\#E(w) - a_\ell(f)) \quad \text {if $E(w)$ is non-singular},\\
            \Norm_{K_f/\Q}(a_\ell(f)^2-(\ell+1)^2) \quad \text{if $E(w)$ is singular}.
        }
    \end{equation*}
    Further, if $Q$ is any set of primes such that $\forall \ell\in Q$, $\ell\nmid N$ for all $N\in\mathcal N$, then 
    \begin{equation*}
        p\in \bigcup_{N\in\mathcal N}\bigcup_{f\in S_2(N)}\bigcap_{\ell\in Q}\bigcup_{w\in W_\ell}\left\{p \text{ prime} : p\mid B_\ell(f,w)\text{ or } p=\ell\right\}.
    \end{equation*}
\end{lemma}

\begin{proof}
    If $(a,b,p)$ is a solution to \eqref{GFE} then by the above discussion, there is a newform form $f$ of weight 2 with level $N\in\mathcal{N}$ such that 
    \begin{equation*}
        E:Y^2=\case{ 
            X(X+1)(X+Aa^p) & \text{if } v_2(A)>0, \\
            X(X+1)(X-Bb^p) & \text{if } v_2(B)>0.}
    \end{equation*}
    We have $(\Tilde{a},\Tilde{b})\in W_\ell$, where $\Tilde{a}$ and $\Tilde{b}$ are the reductions of $a$ and $b$ respectively mod $\ell$. Let $w=(\Tilde{a},\Tilde{b})$. Then 
    \begin{equation}
        \Tilde{E}(\F_\ell)=E(w).
    \end{equation}
    Further, if $E(w)$ is non-singular, then  $a_\ell(E)\equiv a_\ell(f)\mod{\mathfrak p}$, or $a_\ell(f)^2\equiv (\ell+1)^2\mod{\mathfrak p}$ if $E(w)$ is singular. In either case, the congruence holds if and only if $p\mid B_\ell(f,w)$. Since $N(\bar\rho_{E,p})\in\mathcal N$, we thus have that 
    \[
        p\in \bigcup_{N\in\mathcal N}\bigcup_{f\in S_2(N)}\bigcup_{w\in W_\ell} \left\{p \text{ prime} : p\mid B_\ell(f,w)\right\}.
    \]
    Finally, since this holds for any prime $\ell\nmid\prod_{N\in\mathcal N}N$, it holds for the intersection across arbitrarily many such $\ell$.
\end{proof}

In particular, if for all possible levels $N\in\mathcal{N}$ and all weight 2 newforms $f$ of level $N$, there exists $\ell\nmid N$ prime such that $\forall w\in W_\ell$, $\{p\text{ prime}:p\mid B_\ell(f,w)\}$ is finite, then we have that $p$ is one of only finitely many possibilities.

\begin{corollary}
    \label{system-p-values}
    Suppose we have a system of three equations, each of the form as in Lemma \ref{p-values},
    \begin{align}
        A_0x^p-B_0y^p &= 1 \label{system-0}\\
        A_1z^p-B_1y^p &= 1 \label{system-1}\\
        A_2x^p-B_2z^p &= 1. \label{system-2}
    \end{align}
    Let $\mathcal N^{(i)}$ denote the set of possible levels of $A_i,B_i$, and $\mathcal N=\mathcal N^{(0)}\cup\mathcal N^{(1)}\cup\mathcal N^{(2)}$. Let $\ell$ be a prime such that $\ell\nmid N$ for all $N\in\mathcal N$. Let $W_\ell$ be the set of triples $w=(w_0,w_1,w_2)\in(\Z/\ell\Z)^3$ such that the system \eqref{system-0}-\eqref{system-2} is satisfied when viewed modulo $\ell$. Let $B_\ell^{(i)}(f,w)$ be defined as in Lemma \ref{p-values} for equations \eqref{system-0}-\eqref{system-2} respectively. Then for fixed $i=0,1,2$,
    \begin{equation*}
        p\in \bigcup_{N\in\mathcal N^{(i)}}\bigcup_{f\in S_2(N_i)}\bigcap_{\ell\in Q}\bigcup_{w\in W_\ell}\left\{p \text{ prime} : p\mid B_\ell^{(i)}(f,w)\text{ or } p=\ell\right\}.
    \end{equation*}
    Further, let $S_2(N_0,N_1,N_2)=S_2(N_0)\times S_2(N_1)\times S_2(N_2)$. Then,
    \begin{equation*}
        p\in 
        \bigcup_{\substack{(N_0,N_1,N_2)\in\\\mathcal{N}_0\times\mathcal{N}_1\times\mathcal{N}_2}}
        \bigcup_{\substack{(f_0,f_1,f_2)\in\\S_2(N_0,N_1,N_2)}}
        \bigcap_{\ell\in Q}
        \bigcup_{w\in W_\ell}
        \left\{p\text{ prime}:p\mid B_\ell^{(i)} (f_i,w),\forall i=1,2,3\text{ or } p=\ell \right\}.
    \end{equation*}
\end{corollary}

\begin{proof}
    For fixed $i=0,1,2$, the proof of
    \[
        p\in \bigcup_{N\in\mathcal N^{(i)}}\bigcup_{f\in S_2(N_i)}\bigcap_{\ell\in Q}\bigcup_{w\in W_\ell}\left\{p \text{ prime} : p\mid B_\ell^{(i)}(f,w)\text{ or } p=\ell\right\}
    \]
    follows as in the proof of Lemma \ref{p-values}, with the only difference being that $W_\ell$ is the set of solutions to the whole system \eqref{system-0}-\eqref{system-2} when viewed modulo $\ell$ instead of just a single equation. Since this is true for each $i=0,1,2$, it must be simultaneously true for all $i$. Thus
    \begin{align*}
        p&\in\bigcap_{i=0}^2\left(\bigcup_{N\in\mathcal N^{(i)}}\bigcup_{f\in S_2(N_i)}\bigcap_{\ell\in Q}\bigcup_{w\in W_\ell}\left\{p \text{ prime} : p\mid B_\ell^{(i)}(f,w)\text{ or } p=\ell\right\} \right)\\
        &= \bigcup_{\substack{(N_0,N_1,N_2)\in\\\mathcal{N}_0\times\mathcal{N}_1\times\mathcal{N}_2}}
        \bigcup_{\substack{(f_0,f_1,f_2)\in\\S_2(N_0,N_1,N_2)}}
        \bigcap_{\ell\in Q}
        \bigcup_{w\in W_\ell}
        \bigcap_{i=0}^2\left\{p \text{ prime} : p\mid B_\ell^{(i)}(f,w)\text{ or } p=\ell\right\} \\
        &= 
        \bigcup_{\substack{(N_0,N_1,N_2)\in\\\mathcal{N}_0\times\mathcal{N}_1\times\mathcal{N}_2}}
        \bigcup_{\substack{(f_0,f_1,f_2)\in\\S_2(N_0,N_1,N_2)}}
        \bigcap_{\ell\in Q}
        \bigcup_{w\in W_\ell}
        \left\{p\text{ prime}:p\mid B_\ell^{(i)} (f_i,w),\forall i=0,1,2\text{ or } p=\ell \right\}.
    \end{align*}
\end{proof}

\begin{remark}
    \begin{enumerate}
        \item Note that
        \begin{align*}
            &\bigcup_{N\in\mathcal N^{(i)}}\bigcup_{f\in S_2(N_i)}\bigcap_{\ell\in Q}\bigcup_{w\in W_\ell}\left\{p \text{ prime} : p\mid B_\ell^{(i)}(f,w)\text{ or } p=\ell\right\} \\
            &= \bigcap_{\ell\in Q}\bigcup_{N\in\mathcal N^{(i)}}\bigcup_{f\in S_2(N_i)}\bigcup_{w\in W_\ell}\left\{p \text{ prime} : p\mid B_\ell^{(i)}(f,w)\text{ or } p=\ell\right\}.
        \end{align*}
        This similarly holds when considering the intersection across the three equations.
        
        \item Suppose there exists a prime $\ell$ such that $v_\ell(A_0),v_\ell(A_2)>0$, and $v_\ell(A_1B_1)=0$. Then from \eqref{system-0}, we get that
        \begin{equation*}
            1=A_0x^p-B_0y^p\equiv-B_0y^p\mod \ell\Longrightarrow y^p\equiv-B_0\inv\mod \ell.
        \end{equation*}
        Similarly, from \eqref{system-2},
        \begin{equation*}
            z^p\equiv-B_2\inv.
        \end{equation*}
        Thus, there is only one possible solution to \eqref{system-1} when viewed modulo $\ell$. We can get a similar result for equations \eqref{system-0} and \eqref{system-2}. Thus, when computing the possible values of $p$, we will apply the corollary with $Q$ being the set of primes such that there is a unique solution to one of the equations modulo $\ell$.

        \item $Q$ need not be a fixed set. Instead, we can let it depend on the newform $f$. Indeed, fix $i=0,1,2$. Then for any prime $\ell\nmid \prod_{N\in\mathcal N^{(i)}}N$ and some $w\in W_\ell$, we have that $p\mid B_\ell^{(i)}(f,w)$. In particular, let $Q_f$ be a set of such primes $\ell$ depending somehow on the choice of $f$. It is still the case that $p\mid B_q^{(i)}(f,w)$ for any $\ell\in Q_f$ so we get that
        \begin{equation}
            p\in \bigcup_{N\in\mathcal N^{(i)}}\bigcup_{f\in S_2(N_i)}\bigcap_{\ell\in Q_f}\bigcup_{w\in W_\ell}\left\{p \text{ prime} : p\mid B_\ell^{(i)}(f,w)\text{ or } p=\ell\right\}.
        \end{equation}
        Similarly, 
        \begin{equation*}
            p\in 
            \bigcup_{\substack{(N_0,N_1,N_2)\in\\\mathcal{N}_0\times\mathcal{N}_1\times\mathcal{N}_2}}
            \bigcup_{\substack{(f_0,f_1,f_2)\in\\S_2(N_0,N_1,N_2)}}
            \bigcap_{\ell\in Q_{f_0,f_1,f_2}}
            \bigcup_{w\in W_\ell}
            \left\{p\text{ prime}:p\mid B_\ell^{(i)} (f_i,w),\forall i=1,2,3\text{ or } p=\ell \right\}.
        \end{equation*}
        When computing this set, we'll often let $Q_{f_0,f_1,f_2}$ be the set of all suitable primes $\ell$ up to a certain bound, or until the first $\ell$ is found such that the set
        \[
            \bigcup_{w\in W_\ell}
            \left\{p\text{ prime}:p\mid B_\ell^{(i)} (f_i,w),\forall i=1,2,3\text{ or } p=\ell \right\}
        \]
        is finite.
        \item A similar result holds if we pick any two equations. Let $0\le i\neq j\le2$. Then we have that
        \begin{equation*}
            p\in \bigcup_{\substack{(N_i,N_j)\in\\\mathcal{N}_i\times\mathcal{N}_j}}
            \bigcup_{\substack{(f_i,f_j)\in\\S_2(N_i,N_j)}}
            \bigcap_{\ell\in Q}
            \bigcup_{w\in W_\ell}
            \left\{p\text{ prime}:p\mid B_\ell^{(i)} (f_i,w)\text{ and } p\mid B_\ell^{(j)}(f_j,w),\text{ or } p=\ell \right\},
        \end{equation*}
        where $S_2(N_i,N_j)=S_2(N_i)\times S_2(N_j)$.
    \end{enumerate}
\end{remark}

The following result is used to discard many base case equations which have an evident solution.
\begin{theorem}
    \label{DM}
    \cite[Main Theorem]{Darmon-Merel} The equation $x^n+y^n=2z^n$ has no primitive solutions $(a,b,c) \in \Z^3$ when $n\ge3$ unless $abc = 0, \pm 1$.
\end{theorem}

\begin{corollary}
    \label{DM-Corollary}
    Suppose $(x,y,p)$ is a non-trivial solution to \eqref{main-equation}. By Proposition \ref{descent}, $\exists\alpha_i,d_i\in\N$ and $u_i\in\N$ such that
    \begin{align*}
        \alpha_1D_1\frac{u_1^n}{d_1}-\alpha_0D_0\frac{u_0^n}{d_0} &= 1 \\
        \alpha_2D_2\frac{u_2^n}{d_2}-2\alpha_0D_0\frac{u_0^n}{d_0} &= 1 \\
        2\alpha_1D_1\frac{u_1^n}{d_1} - \alpha_2D_2\frac{u_2^n}{d_2} &= 1 \\
        \frac{T_k(x)}{D_0D_1D_2} &= \alpha_3\frac{u_3^n}{d_3}
    \end{align*}
    and 
    \[
        x=\alpha_0D_0\frac{u_0^n}{d_0}, \quad
        x+1=\alpha_1D_1\frac{u_1^n}{d_1}, \quad
        2x+1=\alpha_2D_2\frac{u_2^n}{d_2}.
    \]
    Then we do not have any of the following:
    \begin{enumerate}
        \item $\frac{\alpha_1D_1}{d_1},\frac{\alpha_0D_0}{d_0}\in\N$ and $\frac{\alpha_1D_1}{d_1}\frac{\alpha_0D_0}{d_0}=2$,
        \item $\frac{\alpha_2D_2}{d_2},\frac{\alpha_0D_0}{d_0}\in\N$ and $\frac{\alpha_2D_2}{d_2}\frac{\alpha_0D_0}{d_0}=1$, or
        \item $\frac{\alpha_2D_2}{d_2},\frac{\alpha_1D_1}{d_1}\in\N$ and $\frac{\alpha_2D_2}{d_2}\frac{\alpha_1D_1}{d_1}=1$.
    \end{enumerate}
\end{corollary}

\begin{proof}
    Suppose $\frac{\alpha_1D_1}{d_1},\frac{\alpha_0D_0}{d_0}\in\Z$ and $\frac{\alpha_0\alpha_1D_0D_1}{d_0d_1}=2$. Then the first equation leads to either
    \[
        2u_1^n-u_0^n=1,\quad\text{or}\quad u_1^n-2u_0^n=1
    \]
    By Theorem \ref{DM}, the only solutions are $(u_0,u_1)=(1,1)$ and $(-1,-1)$ respectively. In the former case, we have $x=1$, which is a trivial solution. In the latter case, we have $x=-2$, which is not a valid solution. The remaining two cases follow similarly.
\end{proof}

\section{Local constraints}

To resolve a specific prime exponent $p$ for a generalized Fermat equation, i.e. show there are no non-trivial primitive solutions $(a,b,c)$ for that specific prime exponent $p$, we look for a local obstruction at a prime $\ell\equiv1\mod p$. Indeed, being a $p$th power mod $\ell$ is a very restrictive constraint. For example, if $\ell=2p+1$, then the only $p$th powers modulo $\ell$ are $\pm1$. In general, if $\ell=tp+1$, then there are only $t$ possible $p$th powers.

\begin{lemma}
    \label{mod-l}
    Let $(x,y,p)$ be a solution to \eqref{main}. Then for any prime $\ell$, there exists a solution to the system \eqref{14}--\eqref{16} modulo $\ell$.
\end{lemma}

This follows immediately from Proposition \ref{descent}. Even if one of the equations from the descent has a $p$th power solution mod $\ell$, for some $\ell\equiv1\mod p$, it is very unlikely that such a solution satisfies the entire system.

\begin{lemma}
    \label{local-kraus}
    Let $p$ and $\ell$ be primes such that $\ell=pt+1$ for some $t\in\N$. Let $A,B\in\Z_{(\ell)}$ such that $|v_\ell(A)|,|v_\ell(B)|<p$ and suppose $(x,y,p)$ is a solution to 
    \[
        Ax^p - By^p = 1
    \]
    such that $x,y\in\Z$ and $Ax^p,By^p\in\Z$. If $v_\ell(B)\neq0$, then $A^{-t}\equiv1\mod\ell$.
\end{lemma}

\begin{proof}
    Suppose $v_\ell(B)<0$. Since $By^p\in\Z$, we must have that $v_\ell(y)\ge1$. Then
    \[
       v_\ell(By^p)=v_\ell(B)+pv_\ell(y)>-p + p=0.
    \]
    If $v_\ell(B)>0$, then we immediately have $v_\ell(By^p)>0$. Thus,
    \[
        1=Ax^p-By^p\equiv Ax^p\mod\ell .
    \]
    Then $A\inv\equiv x^p\mod\ell $ and hence $A^{-t}\equiv 1\mod\ell$.
\end{proof}

Note that if we just consider the system \eqref{14}--\eqref{16}, we lose the information from \eqref{17}. 

\begin{lemma}
    \label{local-kraus-with-Tk}
    Suppose $(x,y,p)$ is a solution to \eqref{main-equation}. Then
    \[
        T_k(x) = \frac{C_k}{\alpha_0\alpha_1\alpha_2}\frac{d_0d_1d_2}{D_0D_1D_2} \cdot v^p, \quad \mbox{ for some } v \in \Z
    \]
    where $\alpha_i,D_i,d_i$, $0\le i\le 2$ are as in Proposition \ref{descent}.
\end{lemma}

\begin{proof}
    By Proposition \ref{descent}, there exists $\alpha_0,...,\alpha_3,d_{-1},...,d_3\in\N$ such that 
    \[
        \alpha_0\alpha_1\alpha_2\alpha_3=\frac{C_k}{\gcd(D_0^2D_1^2D_2^2,C_k)},\quad d_{-1}d_0d_1d_2d_3=\frac{D_0^2D_1^2D_2^2}{\gcd(D_0^2D_1^2D_2^2,C_k)},
    \]
    and 
    \[
        x=\alpha_0D_0\frac{u_0^p}{d_0},\quad x+1=\alpha_1D_1\frac{u_1^p}{d_1},\quad 2x+1=\alpha_2D_2\frac{u_2^p}{d_2},\quad\text{and}\quad T_k(x)=D_0D_1D_2\alpha_3\frac{u_3^p}{d_3}.
    \]
    Thus,
    \begin{align*}
        y^p &= \frac{1}{C_k}x(x+1)(2x+1)T_k(x) \\
        &= \frac{1}{C_k}\alpha_0D_0\frac{u_0^p}{d_0}\cdot \alpha_1D_1\frac{u_1^p}{d_1}\cdot\alpha_2D_2\frac{u_2^p}{d_2}\cdot T_k(x)\\
        &= \frac{\alpha_0\alpha_1\alpha_2}{C_k}\frac{D_0D_1D_2}{d_0d_1d_2}(u_0u_1u_2)^pT_k(x)
    \end{align*}
   and hence
    \[
    T_k(x)= \frac{C_k}{\alpha_0\alpha_1\alpha_2}\frac{d_0d_1d_2}{D_0D_1D_2}\left(\frac{y}{u_0u_1u_2}\right)^p.
    \]
    If we let $v=\frac{y}{u_0u_1u_2}$, then we have the result.
\end{proof}

Note that if we solve the system \eqref{14}--\eqref{16} modulo a prime $\ell$, then there is a small finite list of possible values of $x$ modulo $\ell$, and thus a finite set of values for $T_k(x)$ modulo $\ell$. We can compare this list to the set of values 
\[
    \left \{\frac{C_k}{\alpha_0\alpha_1\alpha_2}\frac{d_0d_1d_2}{D_0D_1D_2} \cdot v^p:v\in\F_\ell\right\},
\]
and if there are no matching values, then we have found an obstruction.

It is sometimes the case that the system \eqref{14}--\eqref{17} is consistent modulo $\ell$ for every $\ell\equiv1\mod p$ in a given range. In this case, we will try to view the system modulo $p^2$ instead. 

\begin{remark}
    \label{p-squared}
    Let $p$ be a prime and let $A,B\in\Q^\times$ such that $|v_\ell(A)|,|v_\ell(B)|<p$. Suppose $(x,y,p)$ is a solution to 
    \[
        Ax^p-By^p = 1
    \]
    such that $x,y\in\Z$ and $Ax^p,By^p\in\Z$. If $v_\ell(A)<0$, then $\left(\frac{x}{p},y,p\right)$ is an integral solution to
    \[
        Ap^pX^p - BY^p=1.
    \]
    Indeed, since $Ax^p\in\Z$, we must have that $v_\ell(x)>0$ so that $v_\ell\left(\frac{x}{p}\right)\ge0$ and 
    \[
        Ap^p\left(\frac{x}{p}\right)^p-By^p=Ax^p-By^p=1.
    \]
\end{remark}

In particular, when we want to view the system \eqref{14}--\eqref{16} modulo $p^2$ in order to show that there are no solutions over $\Z$, we will instead show that this `rationalized' equation has no solution modulo $p^2$.

\section{Linear forms in logarithms}

We will make use of two powerful theorems to deal with most cases not handled by the modular and local methods.

\begin{theorem}
    \label{bennett}
    \cite{bennett} If $A,B,n\in\Z$, $AB\neq0$, $n\ge3$, then 
    \[
        |Aa^n-Bb^n|=1
    \]
    has at most one solution in positive integers $(a,b)$.
\end{theorem}

In particular, if $A=B+1$, then the only solution is $(a,b)=(1,1)$. 

We will more generally consider $A, B\in\Q$ and we would like to get an upper bound 
of the exponent $n$. Suppose that $A, B>0$ and $A\neq \pm B+1$. If $n$ is even, then
we can always find a positive solution $a, b$. If $n$ is odd and $a$ and $b$ are 
of opposite sides, since $A\neq \pm B+1$, so either $|a| \ge 2$ or $|b| \ge 2$, say
$|b|\ge 2$. It follows that
\[
1= |Aa^n -Bb^n| = A|a|^n +B|b|^b \ge B2^n
\]
and hence we have $n \le -\frac{\log B}{\log 2}.$
%
%
Thus, without loss of generality, 
For the remainder of this section, we will assume that $a,b>0$.

Below is a consequence of the linear forms in logarithms of Lemma 4.2 \cite{BGP}. We note further that in our application, we may assume without loss of generality that $a$ and $b$ are of the same sign from the descent step.

\begin{theorem}
    \label{log-bound-general}
    Let $A=\frac{p_1}{q_1}$ and $B=\frac{p_2}{q_2}\in\Q$  be positive rational numbers such that $p_1,p_2,q_1,q_2\in\N$ are pairwise coprime and $n \in \N$ with $n \ge 2$. Suppose there exist integers $a,b$ with $a\ge1$ and $b\ge2$ and 
    \begin{equation*}
        |Aa^n-Bb^n|= 1.
    \end{equation*}
    Let $\rho\in\R$ with $\rho >1$, $n_0=\rho^5$, $\lambda=\log\rho$, $h=\log n$. Let
    \begin{equation*}
        b_1=1, \quad b_2=n, \quad \alpha_1=\frac{A}{B}, \quad \alpha_2=\frac{b}{a},
    \end{equation*}
    \begin{equation*}
        a_1=(\rho-1)|\log\alpha_1|+2h(\alpha_1),\quad a_2=A_2\max\{\log b,1\},
    \end{equation*}
    where $h(x/y) =\max\{ \log |x|, \log |y| \}$ for coprime non-zero integer $x$ and $y$, and
    \begin{equation*}
        A_2=\frac{\rho-1}{n_0}\left(\frac{2}{B}+\left|\log\left(\frac{A}{B}\right)\right|\right) + 2\max\left\{\frac{1}{n_0}\log\left(\frac{1}{A}\left(B+1\right)\right)+1,1\right\}.
    \end{equation*}
    We also define
    \begin{equation*}
        R_1=\frac{4}{\lambda}+\frac{4}{\log n_0}+\frac{\lambda}{\log^2 n_0},\quad S_1 =\frac{1}{\log n_0}+\frac{1}{\lambda},
    \end{equation*}
    and
    \begin{align*}
        B_1 &= \frac{a_1A_2R_1^2}{9\lambda}, \\
        B_2 &= \frac{2R_1a_1}{3\log n_0}+\frac{2A_2R_1}{3\log n_0}, \\
        B_3 &= \frac{16\sqrt{2}}{3}\sqrt{a_1A_2}S_1^{3/2}(\log n_0)^{-1/2}, \\
        B_4 &= \left(\log\left(\frac{a_1A_2S_1^2}{\lambda}\right) + \frac{3}{2}\lambda + \frac{3}{20}\right)\frac{1}{A_3} + \frac{\log A_3}{A_3}+\frac{2}{\log n_0},
    \end{align*}
    where $A_3=\max\{\log b,1\}\log^2n_0$. Further, suppose that
    \begin{equation*}
        a_1+a_2\ge4\max(1,\lambda)\quad\text{and}\quad\frac{1}{a_1}+\frac{1}{a_2}\le\min(1,\lambda\inv) .
    \end{equation*}
    Let $n_1$ be the smallest positive integer such that
    \begin{equation*}
        C <\frac{n_1\log2}{\log^2 n_1},
    \end{equation*}
    where 
    \begin{equation*}
        C = \frac{\log\left(\frac{2}{B}\right)}{A_3}+(B_1+B_2+B_3+B_4)
    \end{equation*}
    and 
    \begin{equation*}
        n_2=\frac{1}{A_2\left(e^{-(\log\lambda+1.56)}-\frac{1}{a_1}\right)}.
    \end{equation*}
    Then $n<\max\{n_0,n_1,n_2\}$.
\end{theorem}

\begin{proof}
    Assume towards a contradiction that $n\ge\max\{n_0,n_1,n_2\}$. From the equation
    \[
        |Aa^n-Bb^n|= 1,
    \]
    we have
    \[
        \left|\left(\frac{A}{B}\right)\left(\frac{a}{b}\right)^n-1\right|=\frac{1}{Bb^n}.
    \]
    Using the inequality $|x|<2|e^x-1|$ when $x \ge -\ln 2$, we have
    \begin{equation}
        \left|\log\left(\frac{A}{B}\right) - n\log\left(\frac{b}{a}\right)\right|<2\left|\left(\frac{A}{B}\right)\left(\frac{a}{b}\right)^n-1\right|=\frac{2}{Bb^n}, \label{log.1}
    \end{equation}
    and so
    \begin{equation*}
        \left|\log\left(\frac{b}{a}\right)\right|<\frac{1}{n}\left(\frac{2}{Bb^n}+\left|\log\left(\frac{A}{B}\right)\right|\right) .
    \end{equation*}
    Thus
    \begin{equation*}
        (\rho-1)\left|\log\left(\frac{b}{a}\right)\right|<\frac{\rho-1}{n}\left(\frac{2}{Bb^n}+\left|\log\left(\frac{A}{B}\right)\right|\right)\le\frac{\rho-1}{n_0}\left(\frac{2}{B}+\left|\log\left(\frac{A}{B}\right)\right|\right).
    \end{equation*}
    Then we have
    \begin{align*}
        a^n=\frac{1}{A}(Bb^n\pm1) \le \frac{1}{A}\left(B+\frac{1}{b^n}\right)b^n \le \frac{1}{A}\left(B+1\right)b^n,
    \end{align*}
    so
    \begin{align}
        \log a\le\frac{1}{n}\log\left(\frac{1}{A}\left(B+1\right)b^n\right) &= \frac{1}{n}\log\left(\frac{1}{A}\left(B+1\right)\right)+\log b  \nonumber\\
        &\le\frac{1}{n_0}\log\left(\frac{1}{A}(B+1)\right)+\max\{\log b,1\}  \label{log.2}.
    \end{align}
    If $\log\left(\frac{1}{A}\left(B+1\right)\right)\le 0$, then $\log a\le \log b$. Otherwise, \eqref{log.2} gives
    \begin{align*}
      \log a  &\le\frac{1}{n_0}\log\left(\frac{1}{A}\left(B+1\right)\right)+\max\{\log b,1\} \\
        &\le \left(\frac{1}{n_0}\log\left(\frac{1}{A}\left(B+1\right)\right)+1\right)\max\{\log b,1\} .
    \end{align*}
    Then we have
    \[
        h(\alpha_2) = \max\{\log a,\log b\} 
        \le \max\left\{ \left(\frac{1}{n_0}\log\left(\frac{1}{A}\left(B+1\right)\right)+1\right),1\right\}\max\{\log b,1\}
    \]
    and
    \begin{align*}
        a_2 &= A_2\max\{\log b,1\} 
        = \frac{\rho-1}{n_0}\left(\frac{2}{B}+\left|\log\left(\frac{A}{B}\right)\right|\right)\max\{\log b,1\} \\
        &\quad+2\max\left\{\frac{1}{n_0}\log\left(\frac{1}{A}\left(B+1\right)\right)+1,1\right\}\max\{\log b,1\} \\
        &\ge \frac{\rho-1}{n_0}\left(\frac{2}{B}+\left|\log\left(\frac{A}{B}\right)\right|\right) + 2h(\alpha_2) > (\rho-1)\left|\log\left(\frac{b}{a}\right)\right|+2h(\alpha_2) .
    \end{align*}
    By definition of $a_1$, we have 
    \[
        a_1=(\rho-1)\left|\log\alpha_1\right|+2h(\alpha_1)=(\rho-1)\left|\log\left(\frac{A}{B}\right)\right|+2h(\alpha_1) .
    \]
    Next, by assumption $n\ge n_2$ so that
    \begin{align*}
        n &\ge \frac{1}{A_2\left(e^{-(\log\lambda+1.56)}-\frac{1}{a_1}\right)} \ge \frac{1}{A_2\max\{\log b,1\}\left(e^{-(\log\lambda+1.56)}-\frac{1}{a_1}\right)}.
    \end{align*}
    By rearranging the terms in the above inequality, we get
    \[
    \frac{n}{\lambda e^{1.56}} \ge \frac{1}{A_2\max\{\log b, 1\}}+\frac{n}{a_1}
    \]
    and hence, we have
    \begin{align*}
        h=\log n &\ge \log\left(\frac{1}{A_2\max\{\log b, 1\}}+\frac{n}{a_1}\right)+\log(\lambda e^{1.56}) \\
          &\ge  \log\left(\frac{b_1}{a_2}+\frac{b_2}{a_1}\right)+\log\lambda+1.56 .  
    \end{align*}
    Thus all the conditions of Lemma 4.2 in \cite{BGP} are satisfied and  we have
    \begin{equation}
        \log\left|\log\alpha_1-n\log\alpha_2\right|\ge-\frac{a_1a_2}{9\lambda}R^2 - \frac{2}{3}(a_1+a_2)R - \frac{16}{3}\sqrt{2a_1a_2}S^{3/2}-\log\left(\frac{a_1a_2}{\lambda}S^2\right)-\frac{3}{2}\lambda-2h-\frac{3}{20},  \label{log.3}
    \end{equation}
    where $R=\frac{4h}{\lambda}+4+\frac{\lambda}{h}$ and $S=1+\frac{h}{\lambda}$. We now estimate $R$ and $S$. We have
    \begin{equation}
        R=\frac{4h}{\lambda}+4+\frac{\lambda}{h} = \left(\frac{4}{\lambda}+\frac{4}{\log n}+\frac{4}{\log^2n}\right)\log n \le \left(\frac{4}{\lambda}+\frac{4}{\log n_0}+\frac{4}{\log^2n_0}\right)\log n = R_1\log n, \label{log.4}
    \end{equation}
    and
    \begin{equation}
        S=1+\frac{h}{\lambda}\le \frac{\log n}{\log n_0}+\frac{\log n }{\lambda}=\left(\frac{1}{\log n_0}+\frac{1}{\lambda}\right)\log n = S_1\log n.  \label{log.5}
    \end{equation}
   We now estimate the right-hand side of \eqref{log.3}. From \eqref{log.4}, we have
    \begin{equation}
        \frac{a_1a_2}{9\lambda}R^2\le \frac{a_1A_2R_1^2}{9\lambda}\max\{\log b,1\}\log^2n = B_1\max\{\log b,1\}\log^2 n,  \label{log.6}
    \end{equation}
    and
    \begin{align}
        \frac{2}{3}(a_1+a_2)R &\le  \frac{2}{3}(a_1+A_2\max\{\log b,1\})R_1\log n \nonumber \\
        &=\frac{2}{3}\left(\frac{a_1R_1}{\log n}+\frac{A_2R_1\max\{\log b,1\}}{\log n}\right)\log^2n \nonumber \\
        &\le \frac{2}{3}\left(\frac{a_1R_1}{\log n_0}+\frac{A_2R_1}{\log n_0}\right)\max\{\log b,1\}\log^2n \nonumber \\
        &= B_2\max\{\log b,1\}\log^2n.  \label{log.7}
    \end{align}
      From \eqref{log.5}, we have
    \begin{align}
        \frac{16}{3}\sqrt{2a_1a_2}S^{3/2} &\le \frac{16}{3}\sqrt{2a_1A_2\max\{\log b,1\}}(S_1\log n)^{3/2}  \nonumber\\
        &= \frac{16}{3}\sqrt{2a_1A_2}\frac{S_1^{3/2}}{\sqrt{\log n}}\max\{\log b,1\}\log^2n \nonumber \\
        &\le \frac{16}{3}\sqrt{2a_1A_2}\frac{S_1^{3/2}}{\sqrt{\log n_0}}\max\{\log b,1\}\log^2n  \nonumber\\
        &= B_3\max\{\log b,1\}\log^2n.  \label{log.8}
    \end{align}
    Note that if $x\ge x_0$, then $\frac{\log x_0}{x_0}x\ge\log x$. In particular, since $n\ge n_0$, $\max\{\log b,1\}\log^2n\ge\max\{\log b,1\}\log^2n_0=A_3$. So
    in view of \eqref{log.5}, we have
    \begin{align}
        & \log\left(\frac{a_1a_2}{\lambda}S^2\right) + \frac{3}{2}\lambda + 2h+\frac{3}{20} \nonumber  \\
        &\le  \log\left(\frac{a_1A_2S_1^2}{\lambda}\max\{\log b,1\}\log^2n\right) +\frac{3}{2}\lambda+2\log n+\frac{3}{20} \nonumber \\
        &= \log \left(\frac{a_1A_2S_1^2}{\lambda}\right)+\log\left(\max\{\log b,1\}\log^2n\right)+\frac{3}{2}\lambda+2\log n+\frac{3}{20} \nonumber \\
        &\le  \log\left(\frac{a_1A_2S_1^2}{\lambda}\right)+\frac{\log(\max\{\log b,1\}\log^2n_0)}{\max\{\log b,1\}\log^2n_0}\max\{\log b,1\}\log^2n +\frac{3}{2}\lambda+2\log n+\frac{3}{20}\nonumber  \\
        &\le \log\left(\frac{a_1A_2S_1^2}{\lambda}\right) + \frac{\log A_3}{A_3}\max\{\log b,1\}\log^2n +\frac{3}{2}\lambda+2\log n+\frac{3}{20} \nonumber \\
        &\le \left(\left(\log\left(\frac{a_1A_2S_1^2}{\lambda}\right)+\frac{3}{2}\lambda+\frac{3}{20}\right)\frac{1}{A_3}+\frac{\log A_3}{A_3}+\frac{2}{\log n_0}\right)\max\{\log b,1\}\log^2n\nonumber  \\
        &=  B_4\max\{\log b,1\}\log^2n.  \label{log.9}
    \end{align}
    Therefore, from \eqref{log.6} - \eqref{log.9} and \eqref{log.1}, \eqref{log.3} implies
    \begin{equation*}
        \log\left(\frac{2}{B|b|^n}\right) \ge \log|\log\alpha_1-n\log\alpha_2|>-(B_1+B_2+B_3+B_4)\max\{\log b,1\}\log^2n 
        \end{equation*}
        and this gives
            \begin{equation*}
       \frac{\log\left(\frac{2}{B}\right)}{\max\{\log b,1\}\log^2n}+B_1+B_2+B_3+B_4 \ge \frac{n\log b}{\max\{\log b,1\}\log^2n},
        \end{equation*}
 Since $n\ge n_0$,
    \begin{align*}
        C=\frac{\log\left(\frac{2}{B}\right)}{A_3}+B_1+B_2+B_3+B_4 &\ge \frac{\log\left(\frac{2}{B}\right)}{\max\{\log b,1\}\log^2n}+B_1+B_2+B_3+B_4 \\
        &\ge \frac{n\log b}{\max\{\log b,1\}\log^2n} \\
        &\ge \frac{n\log 2}{\log^2n}\ge \frac{n_1\log2}{\log^2n_1}>C,
    \end{align*}
    which is a contradiction. Thus if $n\ge\max\{n_0,n_1,n_2\}$, there are no such $a,b$ can exist.
\end{proof}

Theorem \ref{log-bound-general} requires that $b\ge2$. To handle the possibility of $b=1$, we make the following observation.

\begin{lemma}
    \label{b-equals-1}
    Let $A, B=\frac{p_2}{q_2}\in\Q$ where $p_2,q_2\in\Z$ such that $\gcd(p_2,q_2)=1$. Suppose there exists $a,b\in\N$ such that
    \begin{equation*}
        Aa^p - Bb^p = 1
    \end{equation*}
    such that $q_2\mid b^p$. If $|b|=1$ then $q_2=1$ and either
    \begin{enumerate}
        \item $\frac{1+B}{A}$ is a $p$-th power of an integer, or
        \item $\frac{1-B}{A}$ is a $p$-th power of an integer.
    \end{enumerate}
\end{lemma}

\begin{proof}
    Since $q_2\mid b^p=\pm1$, and $q_2\in\N$, we have $q_2=1$. If $b=1$, then
  $Aa^p - Bb^p = 1$ and hence $a^p = \frac{1+Bb^p}{A}=\frac{1+B}{A}\in\Z$.
 So $\frac{1+B}{A}$ is a $p$-th power. A similar argument works for $b=-1$.
\end{proof}

\begin{remark}
    To rule out $b=\pm1$ from being a solution to $Aa^p-Bb^p=1$, by Lemma \ref{b-equals-1} we can check whether $\frac{1\pm B}{A}$ is an integer and a $p$-th power. If not, then $b\neq\pm1$.
\end{remark}

\begin{remark}
    We have that $Ax^p-By^p=1$ has a solution if and only if $By^p - Ax^p =-1$ has a solution. Since the roles of $A$ and $B$ are different in Theorem \ref{log-bound-general}, it is possible that the bound that we get is different depending on which form of the equation we use. Thus, when calculating the bound on $n$, we can take the smaller of the two bounds when swapping the roles of $A$ and $B$, provided that we also have that $q_1\mid a^n$.

    Further, since Proposition \ref{descent} gives a system of three equations, we can apply Theorem \ref{log-bound-general} to each of the three equations and take the smallest of the bounds.
\end{remark}

\begin{lemma}
    \label{zero-sol}
    Suppose $(x,y,p)$ is a non-trivial solution to \eqref{main} and write 
    \[
        x = \alpha_0D_0\frac{u_0^n}{d_0}, \quad x+1 = \alpha_1D_1\frac{u_1^n}{d_1}, \quad 2x+1=\alpha_2D_2\frac{u_2^n}{d_2},
    \]
    where $\alpha_i, D_i,d_i\in\N$ and $u_i\in\Z$ are as given in Proposition \ref{descent} and Remark \ref{descent-symmetry}. Then $u_0u_1u_2\neq0$.
\end{lemma}

\begin{remark}
    \label{log-bound-whole-system}
    Let $(x,y,n)$ be a solution to \eqref{main}. By Proposition \ref{descent}, $\exists\alpha_i,d_i\in\N$ satisfying 
    \begin{align*}
        \alpha_1D_1\frac{u_1^n}{d_1}-\alpha_0D_0\frac{u_0^n}{d_0} &= 1\\
        \alpha_2D_2\frac{u_2^n}{d_2}-2\alpha_0D_0\frac{u_0^n}{d_0} &= 1 \\
        2\alpha_1D_1\frac{u_1^n}{d_1} - \alpha_2D_2\frac{u_2^n}{d_2} &= 1.
    \end{align*}
    Suppose that one of 
    \[
        \frac{1\pm\frac{\alpha_0D_0}{d_0}}{\frac{\alpha_1D_1}{d_1}}, \quad\frac{1\pm\frac{2\alpha_0D_0}{d_0}}{\frac{\alpha_2D_2}{d_2}},\quad \text{or} \quad \frac{1\pm\frac{\alpha_2D_2}{d_2}}{\frac{2\alpha_1D_1}{d_1}}
    \]
    are not $n$-th power integers. Then by Theorem \ref{log-bound-general}, Lemma \ref{b-equals-1}, and Lemma \ref{zero-sol}, we can achieve a bound on $n$ for the system \eqref{14}-\eqref{16}. 
\end{remark}

\section{The even cases $k\le 42$ and $n$ prime}

Since Jacobson, Pint\'er, and Walsh \cite{walsh} have already verified Sch\"affer's conjecture when $k\le 58$ is even and $p=2$, we may assume that $p\ge3$ is an odd prime.

By Proposition \ref{descent}, for a fixed even $k$, if we can show that for all possible combinations of $\alpha_i,D_i,d_i$, $0\le i\le 3$ have no solutions, then there are no solutions to \eqref{main}. Fix $2\le k\le42$ even and $\alpha_i,D_i,d_i$. If this case satisfies Corollary \ref{DM-Corollary}, then we move on. Otherwise, we attempt to apply Corollary \ref{system-p-values} and the following remark to the system with using the set of auxiliary primes 
\[ Q=\{p\text{ prime} : p< 1000\}. \] 
The modular method, if it succeeds, produces a finite list of possible exceptional exponents $p$.
Next, we apply Theorem \ref{log-bound-general}, Lemma \ref{b-equals-1}, and the following remarks to achieve an upper bound on $p$. To any remaining primes $p$, we apply Lemma \ref{local-kraus}, Remark \ref{p-squared}, and Lemma \ref{local-kraus-with-Tk}. If a system escapes these methods, then we appeal to {\tt Magma}'s Thue equation solver to show that there are no solutions to the system. The only exponent where this is needed is $p=3$.

The only possible systems that escape the outlined procedure are those of the form
\begin{align*}
    Au_1^p - (A+1)u_0^p &= 1 \\
    (2A+1)u_2^p-2(A+1)u_0^p &= 1 \\
    2Au_1^p - (2A+1)u_2^p &= 1
\end{align*}

Theorem \ref{bennett} gives us that the only solution is $(u_0,u_1,u_2)=(-1,-1,-1)$. We then verify that such a solution does not satisfy equation \eqref{main}. For example, one such case is $k=4$, $D_0=D_1=D_2=1$, $\alpha_0=3$, $\alpha_1=2$, $\alpha_3=5$. Since $u_0<0$, we get that $x=-\alpha_1u_0^p=2$. However, $S_4(2)=1+2^4=17$, which is not a $n$th power for any $n\ge2$. 

Applying the same strategy for $2\le k\le42$ even, we have show that there are no solutions to the system \eqref{14}--\eqref{17}, and thus by Proposition \ref{descent} there are no solutions to $S_k(x)=y^n$ for $2\le k\le42$ even.

\section{Conclusions}

We have proved the cases $2 \le k \le 42$ even of Sch\"affer's conjecture by systematically applying a more detailed descent, the modular method, linear forms in logarithms, local methods, or Thue equation solvers. Theorem~\ref{bennett} and Theorem~\ref{DM} are used to treat equations which have evident solutions.

The following conjecture of Kraus explains why local methods often succeed in giving a resolution of Sch\"affer's conjecture.

\begin{conj}
  Let $a,b,c$ be non-zero pairwise coprime integers which do not satisfy a linear relation with coefficients in $\left\{ -1, 0, 1 \right\}$. There is a constant $g(a,b,c)$ such that for all primes $p > g(a,b,c)$, the equation
  \begin{equation}
  \label{abc-fermat}
      a x^p + b y^p + c z^p = 0
  \end{equation}
  has a local obstruction.
\end{conj}

For \eqref{abc-fermat} to have a local obstruction at a prime $q \nmid p abc$, it is necessary that $q \equiv 1 \pmod p$ and $q < \left( (p-1)(p-2) \right)^2$.

Our use of local constraints on a coupled system of ternary equations suggests that the constant $g(a,b,c)$ is quite small and thus on one hand local methods succeed to reach the linear forms in logarithm bounds on the exponent $p$ and on the other hand the small exponents that escape can be treated using Thue equation solvers. Applying the modular method and linear forms in logarithms to a coupled system of ternary equations also makes these methods more efficient, for instance usually one of the equations has a better bound from linear forms in logarithms.

\section{Appendix}

The {\tt Magma} program used can be found in the author's github: \url{https://github.com/JamesHoulee/Schaeffer}.

\bibliography{paper}{}
\bibliographystyle{alpha}

\end{document}